\documentclass[12pt]{amsart}
\usepackage{amsmath,latexsym,amsfonts,amssymb,amsthm}
\usepackage{amssymb}
\usepackage{geometry}
\usepackage{hyperref}
\usepackage{mathrsfs}
\usepackage{graphicx,color}
\usepackage{tikz-cd}
\usepackage{tikz}
\usetikzlibrary{positioning}
\usepackage{mathtools}
\usepackage[all,cmtip]{xy}

\newtheorem{prop}{Proposition}[section]
\newtheorem{thm}[prop]{Theorem}
\newtheorem{cor}[prop]{Corollary}
\newtheorem{conj}[prop]{Conjecture}
\newtheorem{lem}[prop]{Lemma}

\theoremstyle{definition}
\newtheorem{que}[prop]{Question}
\newtheorem{defn}[prop]{Definition}
\newtheorem{expl}[prop]{Example}
\newtheorem{rem}[prop]{\it Remark}

\newtheorem*{claim*}{Claim}

\newcommand{\bP}{\mathbb{P}}
\newcommand{\bC}{\mathbb{C}}
\newcommand{\bR}{\mathbb{R}}
\newcommand{\bA}{\mathbb{A}}
\newcommand{\bQ}{\mathbb{Q}}
\newcommand{\bZ}{\mathbb{Z}}
\newcommand{\bN}{\mathbb{N}}

\newcommand{\tD}{\widetilde{D}}
\newcommand{\tM}{\tilde{M}}
\newcommand{\tGamma}{\tilde{\Gamma}}
\newcommand{\tDelta}{\tilde{\Delta}}

\newcommand{\cX}{\mathcal{X}}

\newcommand{\cO}{\mathcal{O}}

\newcommand{\cI}{\mathcal{I}}
\newcommand{\cM}{\mathcal{M}}

\newcommand{\cJ}{\mathcal{J}}
\newcommand{\cE}{\mathcal{E}}

\newcommand{\fm}{\mathfrak{m}}

\newcommand{\QSing}{\mathrm{QSing}}

\newcommand{\mult}{\mathrm{mult}}
\newcommand{\lct}{\mathrm{lct}}

\newcommand{\Vol}{\mathrm{Vol}}
\newcommand{\ord}{\mathrm{ord}}

\newcommand{\lnlc}{\ell_{\mathrm{nlc}}}

\newcommand{\Aut}{\mathrm{Aut}}
\newcommand{\Bir}{\mathrm{Bir}}
\newcommand{\wt}{\mathrm{wt}}

\newcommand{\rin}{\mathrm{in}}
\newcommand{\Bsr}{\mathrm{Bsr}}
\newcommand{\origin}{\mathbf{0}}

\begin{document}

\title[Birational superrigidity is not a locally closed property]{Birational superrigidity is not a locally closed property (Appendix: On a conjecture of Tian)}

\author{Ziquan Zhuang}
\address{Department of Mathematics, Princeton University, Princeton, NJ 08544, USA.}
\email{zzhuang@math.princeton.edu}
\date{}

\maketitle

\begin{abstract}
    We prove an optimal result on the birational rigidity and K-stability of index $1$ hypersurfaces in $\mathbb{P}^{n+1}$ with ordinary singularities when $n\gg 0$ and also study the birational superrigidity and K-stability of certain weighted complete intersections. As an application, we show that birational superrigidity is not a locally closed property in moduli. We also prove (in the appendix) that the alpha invariant function is constructible in some families of complete intersections.
\end{abstract}

\section{Introduction}

A Fano variety $X$ is said to be birationally superrigid if every birational map $f:X\dashrightarrow Y$ from $X$ to the source of a Mori fiber space is an isomorphism (see e.g. \cite{Che-survey} for an introduction). In particular, such Fano varieties are not rational. This property identifies a very special class of Fano varieties and an interesting question is whether the set of birationally superrigid Fano varieties form a ``nice" moduli. Indeed, it follows from the definition that $X$ has terminal singularities and therefore by the seminal work of Birkar \cite{Birkar-1,Birkar-2} on the Borisov-Alexeev-Borisov conjecture, birationally superrigid Fano varieties belong to a bounded family. Moreover, such moduli (if exists) satisfies the valuative criterion of separatedness by a recent result of Stibitz and the author \cite{SZ-rigid-imply-stable}. In addition, birational superrigidity is a constructible condition by \cite[Corollary 7.8]{SC-log-model}. So the next step is to check whether birational superrigidity is a locally closed property or not. 

Unfortunately the answer is no in general, and one of the main purposes of this note is to construct such counterexamples. We remark that the stronger expectation that birational superrigidity is open in moduli is known to be false, at least in dimension three: one such counterexample is given by the family of quasi-smooth quintic hypersurfaces in $\bP(1^4,2)$ \cite[Example 6.3]{CP-3fold-hypersurface}. The analogous openness question for birational rigidity (a weaker notion than birational superrigidity) \cite[Conjecture 1.4]{Corti-4n^2-ineq} also has a negative answer by the 3-dimensional counterexample in \cite{CG-rigidity-not-open}. In both examples, birational (super)rigidity is a locally closed property in the corresponding moduli.

Our construction is based on the degeneration of hypersurfaces into double covers. Let $m$ be a sufficiently large integer and let $x_0,\cdots,x_{n+1},y$ be the weighted homogeneous coordinates of $\bP(1^{n+2},m)$. Let $f_s,g_s$ (parameterized by $s\in\bA^1$) be homogeneous polynomials in $x_0,\cdots,x_{n+1}$ of degree $2m$ and $m$ respectively so that
\[\cX=(y^2-f_s=ty-g_s=0)\subseteq \bP(1^{n+2},m)\times \bA^2_{s,t}\]
defines a family of weighted complete intersections of dimension $n$ parameterized by $\bA^2$. For $t\neq 0$, it is easy to see that $\cX_{s,t}\cong (t^2f_s-g^2_s=0)\subseteq \bP^{n+1}$ is a hypersurface of degree $2m$ while $\cX_{s,0}$ is the double cover of the hypersurface $G_s=(g_s=0)\subseteq \bP^{n+1}$ branched over the divisor $F_s\cap G_s$ where $F_s=(f_s=0)$. We show that with suitable choices of $f_s$ and $g_s$, this provides the counterexample we want in every sufficiently large odd dimension.

\begin{thm} \label{thm:counterexample}
Notation as above. Let $x\in\bP^{n+1}$. Assume that $n=2m-1$, $m\gg 0$ and the following:
    \begin{enumerate}
        \item $F_0$ and $G_0$ have a unique ordinary singularity at $x$ with $\mult_x F_0=2m-2$ and $\mult_x G_0=m-1$ and are otherwise smooth,
        \item the projective tangent cone of $F_0\cap G_0$ at $x$ is a smooth complete intersection,
        \item $\cX_{s,t}$ is smooth when $s\neq 0$.
    \end{enumerate}
Then $\cX_{s,t}$ is birationally superrigid if and only if $s\neq 0$ or $(s,t)=(0,0)$.
\end{thm}

The study of the birational superrigidity of $\cX_{s,t}$ in the above example can be further divided into two parts. We first prove the birational superrigidity of smooth double covers of Fano index $1$ (which include $\cX_{s,0}$ for $s\neq 0$) by generalizing the multiplicity bound in the complete intersection case \cite{Puk-hypersurface,Suzuki-cpi}. Indeed we prove something stronger:

\begin{thm} \label{thm:weighted cpi}
Fix $r,s\in \bZ$, then there exists an integer $M$ such that every smooth weighted complete intersection $X\subseteq\bP(1^m,a_1,\cdots,a_s)$ of codimension $r$ and index one $($i.e. $-K_X\sim H:=c_1(\cO_X(1)))$ with a base-point-free anticanonical linear system $|-K_X|$ is birationally superrigid and K-stable when $m\ge M$.
\end{thm}

We refer to \cite{Tian-K-stability-defn,Don-K-stability-defn} for the definition of K-stability in the above statement. In our cases, K-stability is a direct consequence of the birational superrigidity by a simple application of \cite{SZ-rigid-imply-stable}. Note that Theorem \ref{thm:weighted cpi} generalizes the results of \cite{Puk-double,Puk-iterated} by allowing more general weighted complete intersections and removing the generality assumptions (albeit at the cost of increasing dimensions).

Next we handle singularities that may appear on $\cX_{\origin}$ (where $\origin=(0,0)\in\bA^2$). To do so we resolve the singularities by weighted blowups and analyze the maximal singularities on the resolution. As an illustration and application of this technique, we generalize a result of \cite{LZ-singular-cpi} and prove an optimal result about the birational rigidity and K-stability of index $1$ hypersurfaces with ordinary singularities in large dimension:

\begin{thm} \label{thm:hypersurface}
Let $X\subseteq \bP^{n+1}$ be a hypersurface of degree $n+1$ and dimension $n\ge 250$ with only isolated ordinary singularities $($i.e. the projective tangent cones are smooth$)$ of multiplicities at most $m$. Then
    \begin{enumerate}
        \item $X$ is K-stable if $m\le n$;
        \item $X$ is birationally superrigid if $m\le n-2$;
        \item $X$ is birationally rigid if $m\le n-1$. Moreover, linear projection from each point $x\in X$ of multiplicity $n-1$ induces a birational involution $\tau_x$ and the birational automorphism group $\Bir(X)$ of $X$ is generated by $\Aut(X)$ together with these $\tau_x$. 
    \end{enumerate}
\end{thm}

This paper is organized as follows. Birational superrigidity can be characterized by the singularities of movable boundaries and its openness is closely related to the semi-continuity of higher codimensional alpha invariants recently introduced by \cite{Zhu-higher-alpha}, thus in \S \ref{sec:higher alpha} we study some properties of these invariants which may be of independent interest. Theorems \ref{thm:weighted cpi} and \ref{thm:hypersurface} are proved in \S \ref{sec:weighted cpi} and \S \ref{sec:hypersurface} respectively. In \S \ref{sec:counterexample}, we prove Theorem \ref{thm:counterexample} and present a few further questions. Finally in the appendix, we prove that the alpha invariant function is constructible in some special cases.

\subsection*{Notation and conventions}

We work over $\bC$. Unless otherwise specified, all varieties are assumed to be normal and divisors are understood as $\bQ$-divisor. A \emph{boundary} on a variety $X$ is defined as an expression of the form $a\cM$ where $a\in\bQ$ and $\cM$ is a linear system on $X$ (this includes the case when $\cM$ is just a divisor). It is said to be \emph{movable} (resp. $\bQ$-Cartier) if $\cM$ is movable (resp. $\bQ$-Cartier). A \emph{pair} $(X,\Delta)$ consists of a variety $X$ and an effective divisor $\Delta\subseteq X$ such that $K_X+\Delta$ is $\bQ$-Cartier. The notions of terminal, canonical, klt and log canonical (lc) singularities are defined in the sense of \cite{Kol-sing-of-pair}. Let $(X,\Delta)$ be a pair and $M$ a $\bQ$-Cartier boundary (resp. an ideal sheaf, a subscheme) on $X$, the \emph{log canonical threshold} of $M$ with respect to $(X,\Delta)$ is denoted by $\lct(X,\Delta;M)$ (or simply $\lct(X;M)$ when $\Delta=0$).

\subsection*{Acknowledgement}

The author would like to thank his advisor J\'anos Koll\'ar for constant support, encouragement and numerous inspiring conversations. He also wishes to thank Hamid Ahmadinezhad, Ivan Cheltsov, Takuzo Okada, Jihun Park and Vyacheslav V. Shokurov for helpful discussions and for pointing out several useful references and Asher Auel, Yuchen Liu and Ziwen Zhu for helpful conversations.

\section{Higher codimensional alpha invariants} \label{sec:higher alpha}

In this section, we study properties of higher codimensional alpha invariants and construct examples to show that they are not lower semi-continuous in general. This is of independent interest and the results are not used until the end of the note, so readers who are mainly interested in the proof of our main theorems may feel free to skip this section.

\begin{defn} \label{defn:alpha^(k)}
Let $(X,\Delta)$ be a klt pair and $L$ an ample line bundle on $X$. Let $1\le k \le n=\dim X$ be an integer. We define the alpha invariant of codimension $k$ for $(X,\Delta;L)$ to be
\begin{equation} \label{eq:alpha^(k)}
    \alpha^{(k)}(X,\Delta;L)=\inf_{m\ge 1} \alpha_m^{(k)}(X,\Delta;L)
\end{equation}
where $\alpha_m^{(k)}(X,\Delta;L)$, the $m$-th alpha invariant of codimension $k$, is defined as the infimum of $\lct(X,\Delta;\frac{1}{m}\cM)$ over all sub linear systems $\cM\subseteq |mL|$ whose base locus has codimension at least $k$.
\end{defn}

In particular, when $k=1$ this reduces to Tian's alpha invariant \cite[\S 5]{Tian-alpha-defn} and when $X$ is $\bQ$-Fano (i.e. $X$ is klt and $-K_X$ is ample), $\Delta=0$ and $L=-K_X$ (the definition of alpha invariants also makes sense when $L$ is only a $\bQ$-line bundle), this reduces to \cite[Definition 1.1]{Zhu-higher-alpha} and in this case we denote $\alpha^{(k)}(X,\Delta;L)$ by $\alpha^{(k)}(X)$. It is also not hard to see that the infimum in \eqref{eq:alpha^(k)} can be replaced by taking the limit. We note that this convergence $\alpha_m^{(k)}\rightarrow \alpha^{(k)}$ is actually uniform in any given family:

\begin{prop} \label{prop:uniform convergence of alpha}
Let $n,r>0$ be integers and let $\lambda,\epsilon>0$. Then there exists an integer $M$ depending only on $n,r,\lambda$ and $\epsilon$ with the following property: for any klt pair $(X,\Delta)$ and any ample line bundle $L$ such that $X$ has dimension $n$, $rL$ is globally generated, $rL-(K_X+\Delta)$ is ample and $\alpha^{(n)}(X,\Delta;L)<\lambda$, we have
\[\alpha^{(k)}(X,\Delta;L)\le \alpha_m^{(k)}(X,\Delta;L)\le \alpha^{(k)}(X,\Delta;L)+\epsilon\]
for all $1\le k \le n$ and all $m\ge M$.
\end{prop}

\begin{proof}
The argument is essentially the same as \cite[Theorem 5.1]{BL-delta-lsc}, and we reproduce the proof for reader's convenience. 
Let $1\le k\le n$ and let $\alpha=\alpha^{(k)}(X,\Delta;L)$. It is clear from the definition that $\alpha^{(k)}(X,\Delta;L)\le \alpha^{(n)}(X,\Delta;L)$, thus $\alpha<\lambda$ by assumption. Again by definition, there exists an integer $\ell$ and a linear system $\cM\subseteq|\ell L|$ whose base locus has codimension at least $k$ such that
\[\alpha\le \lct(X,\Delta;\frac{1}{\ell}\cM) \le \alpha+\frac{\epsilon}{2}.\]
Let $E$ be an exceptional divisor over $X$ that computes $\lct(X,\Delta;\frac{1}{\ell}\cM)$ and let $A=a(E;X,\Delta)+1$ be the log discrepancy of $E$ with respect to the pair $(X,\Delta)$, then we have $\ord_E(\frac{1}{\ell}\cM) \ge \frac{A}{\alpha+\frac{\epsilon}{2}}$ and hence
\begin{eqnarray*}
\ord_E \cJ(X,\Delta+\frac{m}{\ell}\cM) & \ge & \lfloor \ord_E(\frac{m}{\ell}\cM) - a(E;X,\Delta) \rfloor \\
 & > & \ord_E(\frac{m}{\ell}\cM) - A \\
 & \ge & A\cdot \frac{m-\alpha-\frac{\epsilon}{2}}{\alpha+\frac{\epsilon}{2}}.
\end{eqnarray*}
where $\cJ(X,\Delta+\frac{m}{\ell}\cM)$ denotes the multiplier ideal of the pair. Since $rL-(K_X+\Delta)$ is ample and $\frac{m}{\ell}\cM\sim_\bQ mL$, we have $H^i(X,\cJ(X,\Delta+\frac{m}{\ell}\cM)\otimes\cO_X((m+r)L))=0$ by Nadel vanishing. As $rL$ is globally generated, we see that $\cJ(X,\Delta+\frac{m}{\ell}\cM)\otimes\cO_X((m+r+nr)L)$ is also globally generated by Castelnuovo-Mumford regularity (see e.g. \cite[\S 1.8]{Laz-positivity-1}). Note that the co-support of $\cJ(X,\Delta+\frac{m}{\ell}\cM)$ has codimension at least $k$ since the same is true for the base locus of $\cM$. Therefore, $\cM'=|\cJ(X,\Delta+\frac{m}{\ell}\cM)\otimes\cO_X((m+r+nr)L)|$ defines a sub linear series of $|(m+r+nr)L|$ whose base locus has codimension at least $k$ and from the previous calculations we have
\begin{eqnarray*}
\alpha_{m+n+nr}^{(k)}(X,\Delta;L) & \le & \lct(X,\Delta;\frac{1}{m+n+nr}\cM') \\
 & \le & \frac{(m+r+nr)A}{\ord_E \cJ(X,\Delta+\frac{m}{\ell}\cM)} \\
 & < & \frac{m+r+nr}{m-\alpha-\frac{\epsilon}{2}} (\alpha+\frac{\epsilon}{2})
\end{eqnarray*}
It is clear that there exists $M$ depending only on $n,r,\lambda$ and $\epsilon$ such that the last term of the above inequality is smaller than $\alpha+\epsilon$ whenever $m\ge M$ and $\alpha<\lambda$. Replacing $M$ by $M+r+nr$, it follows that $\alpha_m^{(k)}(X,\Delta;L)<\alpha+\epsilon$ for all $1\le k \le n$ and $m\ge M$.
\end{proof}

\begin{cor} 
Let $f:(X,\Delta)\rightarrow T$ be a $\bQ$-Gorenstein flat family of klt pairs $($in the sense of e.g. \cite[2.3]{BL-delta-lsc}$)$ and let $L$ be an $f$-ample line bundle on $X$. Let $\epsilon>0$. Then there exists an integer $M>0$ such that
\[0\le \alpha_m^{(k)}(X_t,\Delta_t;L_t)-\alpha^{(k)}(X_t,\Delta_t;L_t)\le \epsilon\]
for all $m\ge M$ and $t\in T$.
\end{cor}

\begin{proof}
Since $L$ is $f$-ample, we may choose $r\in\bZ_{>0}$ such that $rL_t$ is very ample and $rL-(K_X+\Delta)$ is $f$-ample for all $t\in T$. It is then not hard to see that $\alpha^{(n)}(X_t,\Delta_t;L_t)\le nr$ ($\forall t\in T$) since for any smooth point $x\in X_t\backslash\Delta_t$, the linear system $\cM=|rL_t\otimes \mathfrak{m}_x|$ is base point free outside $x$ and $\lct(X_t,\Delta_t;\cM)=n$. The result then follows from Proposition \ref{prop:uniform convergence of alpha} by taking $\lambda=nr+1$.
\end{proof}

In particular, as the function $t\mapsto \alpha_m(X_t,\Delta_t;L_t)$ on $T$ is lower semi-continuous for each $m\ge 1$, we obtain the lower semi-continuity of the alpha invariant function \cite[Theorem B]{BL-delta-lsc}. In contrast, the higher codimensional variants $\alpha^{(k)}(X_t,\Delta_t;L_t)$ are in general not lower semi-continuous, as suggested by the following statement (here $\epsilon(L)$ is the Seshadri constant of $L$ at a very general point; see \cite[\S 5]{Laz-positivity-1} for the definition and elementary properties of Seshadri constants):

\begin{prop} \label{prop:alpha and seshadri}
Let $X$ be a smooth projective variety of dimension $n$ and $L$ an ample line bundle on $X$. Then $\alpha^{(n)}(L)\ge \frac{n}{\sqrt[n]{(L^n)}}$ with equality if and only if $\epsilon(L)=\sqrt[n]{(L^n)}$.
\end{prop}

\begin{rem} \label{rem:non lsc example}
Before we prove this proposition, let us explain why it leads to the failure of lower semi-continuity for higher codimensional alpha invariants. Recall that we always have $\epsilon(L)\le \sqrt[n]{(L^n)}$ and the function $t\mapsto \epsilon(L_t)$ is lower semi-continuous in any polarized flat family $f:(X,L)\rightarrow T$. Thus if $f:X\rightarrow T$ is smooth and $L$ is an $f$-ample line bundle such that $\epsilon(L_t)=\sqrt[n]{(L_t^n)}$ at very general $t\in T$ while $\epsilon(L_0)<\sqrt[n]{(L_0^n)}$ on some special fiber $X_0$, then by the above proposition we have
\[\alpha^{(n)}(L_0)> \frac{n}{\sqrt[n]{(L_0^n)}} = \alpha^{(n)}(L_t)\]
and therefore the function $t\mapsto \alpha^{(n)}(L_t)$ is not lower semi-continuous. As a concrete example, let $f:(X,L)\rightarrow T$ be the family of K3 surfaces of degree $16$ (c.f. \cite{Mukai-K3}), then a general member $X_t$ of this family has Picard number one, hence $\epsilon(L_t)=4$ by \cite[Corollary]{Knu-K3-Seshadri}; on the other hand, let $X_0$ be a quartic K3 surface containing a line $\ell$, let $H$ be the hyperplane class and let $C\sim H-\ell$ be the class defining the elliptic fibration, then $L_0=H+2C$ has degree $16$ but $\epsilon(L_0)\le (L_0\cdot C)=3<\sqrt{(L^2)}$.
\end{rem}

\begin{proof}[Proof of Proposition \ref{prop:alpha and seshadri}]
We first prove that $\alpha_m^{(n)}(L)\ge \frac{n}{a}$ for all $m\ge 1$, where $a=\sqrt[n]{(L^n)}$, the inequality for $\alpha^{(n)}(L)$ then follows by taking the limit. Let $\cM\subseteq |mL|$ be a linear system whose base locus supported on isolated points, then we can define the complete intersection subscheme $Z=\cM^n$ by choosing $n$ general member of $\cM$. We claim that $\lct(X;Z)\ge \frac{n}{ma}$. Indeed, let $x\in X$, then we have $\mult_x Z = \ell(\cO_Z) \le (M^n)=m^n a^n$, hence by \cite[Theorem 0.1]{dFEM-mult-and-lct} we have $\lct_x(X;Z)\ge \frac{n}{ma}$ ($\lct_x$ denotes the log canonical threshold at the point $x$) for all $x$ and the claim follows. We then have $\lct(X;\cM)\ge \lct(X;Z)\ge \frac{n}{ma}$ and as $\cM$ is arbitrary, we obtain $\alpha_m^{(n)}(L)\ge \frac{n}{a}$.

Suppose that $\epsilon(L)=a$, then there exists $x\in X$ such that $\epsilon(L;x)=a$ and by definition, $\pi^*L-tE$ is ample for all $0<t<a$ (where $\pi:Y\rightarrow X$ is the blowup of $X$ at $x$ with exceptional divisor $E$). It follows that for $t\in (0,a)\cap \bQ$ and sufficiently large and divisible $m$ (depending on $t$), the linear system $\cM_{m,t}=\pi_*|m(\pi^*L-tE)|=|mL\otimes \mathfrak{m}_x^{mt}|\subseteq |mL|$ is base point free away from $x$ and we get $\alpha^{(n)}(L)\le m\cdot \lct(X;\cM_{m,t})\le m\cdot \frac{n}{mt}=\frac{n}{t}$. Letting $t\rightarrow a$ we see that $\alpha^{(n)}(L)\le\frac{n}{a}$. Since we always have the inequality in the other direction, we obtain the equality $\alpha^{(n)}(L)=\frac{n}{a}$.

Now suppose that $\alpha^{(n)}(L)=\frac{n}{a}$. Then by definition, for each $0<\epsilon_0\ll 1$, there exist some $m\ge 1$ together with some linear system $\cM\subseteq|mL|$ with isolated base points such that $\lct(X,\frac{1}{m}\cM)<\frac{n}{a}+\epsilon_0$. Let $\cI$ be base ideal of $\cM$ (i.e. image of the evaluation map $\cM\otimes L^{-m}\rightarrow \cO_X$), then we have $\lct_x(X;\cI^{1/m})<\frac{n}{a}+\epsilon_0$ for some $x\in X$ and $\mult_x \cI\le m^n(L^n)=m^n a^n$. In particular, $\lct_x(X;\cI)^n\cdot \mult_x \cI\le n^n(1+\epsilon_1)$ for some $0<\epsilon_1\ll 1$ (which depends on $\epsilon_0$). By Lemma \ref{lem:refined dFEM equality}, this implies $\cI\subseteq \fm_x^{m(a-\epsilon)}$ (i.e. inclusion holds after raising both sides to a sufficiently divisible power) for some $0<\epsilon\ll 1$ (which again depends on $\epsilon_0$) and in particular, $\mult_x \cM\ge m(a-\epsilon)$ and since $\cM$ only has isolated base points we have $(L\cdot C) = \frac{1}{m}(\cM\cdot C)\ge (a-\epsilon)\mult_x C$ for any curve $C$ containing $x$. By definition of Seshadri constants, this gives \[\epsilon(L)\ge\epsilon(L;x)=\inf_{x\in C} \frac{(L\cdot C)}{\mult_x C}\ge a-\epsilon.\]
Since $\epsilon$ can be arbitrarily small, we obtain $\epsilon(L)\ge a$. But we always have $\epsilon(L)\le a$, thus equality holds as well.
\end{proof}

The following lemma is used in the above proof.

\begin{lem} \label{lem:refined dFEM equality}
Let $\epsilon\in\bQ_+$, then there exists some $\epsilon_1>0$ such that for every smooth variety $X$ of dimension $n$ and every ideal $\cI\subseteq \cO_X$ co-supported at a point $x\in X$ such that $\lct(X;\cI)^n\cdot \mult_x(\cI)\le n^n(1+\epsilon_1)$, we have $\cI\subseteq \mathfrak{m}_x^{n\mu(1-\epsilon)}$ where $\mu^{-1}=\lct(X;\cI)$.
\end{lem}

\begin{proof}
The argument is similar to those of \cite[Theorem 1.4]{dFEM-mult-and-lct}. We may replace $\cI$ by $\cI^\ell$ for some sufficiently divisible $\ell$ and assume that $\mu(1-\epsilon) \in \bZ$. Since the statement is local, we may also assume that $X=\bA^n$ and $x$ is the origin. In particular, we consider $\cI$ as an ideal in $R=\bC[x_1,\cdots,x_n]$. Let $\fm=\fm_x$. Choose a multiplicative order on the monomials in $R$ and deform all powers $\cI^m$ of $\cI$ to monomial ideals $\rin(\cI^m)$, then $\rin(\cI^m)\cdot \rin(\cI^\ell)\subseteq \rin(\cI^{m+\ell})$ and we have $\cI^m\subseteq \fm^\ell$ if and only if $\rin(\cI^m)\subseteq \fm^\ell$ (for any $m,\ell\in \bN$). As in \cite{dFEM-mult-and-lct}, let $\mu(\cI)=\lct(X;\cI)^{-1}$, let $P(\rin(\cI))\subseteq\bR^n$ be the Newton polytope of $\rin(\cI)$ and let $P_\infty=\cup_{r\in\bN}\frac{1}{2^r}P(\rin(\cI^{2^r}))$. Then by \cite{Howald} and the lower semi-continuity of lct, 
\[\mu(P_\infty)=\lim_{r\rightarrow \infty} \frac{\mu(\rin(\cI^{2^r}))}{2^r}\ge \mu\]
where $\mu(P_\infty)=\inf\{t>0\,|\,(t,t,\cdots,t)\in P_\infty\}$. Let $F=(\sum_{i=1}^n a_i u_i=1)$ be a supporting hyperplane of $P_\infty$ through the point $\mu(P_\infty)\cdot (1,1,\cdots,1)$ and let $S_F=\{u\in\bR^n_{\ge0}\,|\,\sum_{i=1}^n a_i u_i\le 1\}$. Then $\mu\cdot (1,\cdots,1)\in S_F$ and we have 
\[\mult_x(\cI) = \lim_{r\rightarrow \infty}\frac{n!\cdot \ell(R/\cI^{2^r})}{(2^r)^n} = \lim_{r\rightarrow \infty}\frac{n!\cdot \ell(R/\rin(\cI^{2^r}))}{(2^r)^n} \ge n!\Vol(S_F)\ge \prod_{i=1}^n \frac{1}{a_i}.\]
By assumption, $\mult_x(\cI)\le (n\mu)^n(1+\epsilon)$, hence $\prod_{i=1}^n \frac{1}{a_i}\le (n\mu)^n(1+\epsilon_1)$, or \[(1+\epsilon_1)\prod_{i=1}^n a_i\ge \left(\frac{\mu^{-1}}{n}\right)^n \ge \left(\frac{\sum_{i=1}^n a_i}{n}\right)^n.\]
Comparing with the AM-GM inequality, we see that if $\epsilon_1>0$ is chosen to be sufficiently small (which only depends on $\epsilon$), then we have $a_i\le \frac{1}{(1-\epsilon)}\cdot \frac{1}{n} \sum_{i=1}^n a_i\le \frac{1}{n\mu(1-\epsilon)}$. It follows that $\Delta(\mu(1-\epsilon)):=\{u\in\bR^n_{\ge0}\,|\,\sum_{i=1}^n u_i\le n\mu(1-\epsilon)\}\subseteq S_F\subseteq \bR^n_{\ge0}\backslash P_\infty$; in other words, $\rin(\cI^{2^r})\subseteq \fm^{2^r n\mu(1-\epsilon)}$ for all $r\in\bN$. Hence $\cI\subseteq \fm^{n\mu(1-\epsilon)}$.
\end{proof}

\section{Weighted complete intersection} \label{sec:weighted cpi}

In this section we prove Theorem \ref{thm:weighted cpi}. Throughout the section, we denote by $\bP$ the weighted projective space $\bP(1^m,a_1,\cdots,a_s)$ and let $x_1,\cdots,x_m,y_1,\cdots,y_s$ be the weighted homogeneous coordinates. Given a weighted complete intersection $X\subseteq \bP$, we also denote by $\QSing(X)$ the subset along which $X$ is not quasi-smooth and let $\delta_X=\dim \QSing(X)$ (by convention, $\dim(\emptyset)=-1$).

\begin{lem}
Let $X\subseteq \bP$ be a weighted complete intersection of codimension $r$ and let $Y=X\cap (y_s=0)$. Then $\delta_Y\le \delta_X+r$.
\end{lem}

\begin{proof}
Let $f_1,\cdots,f_r$ be the defining equations of $X$. We claim that
\[\QSing(Y)\cap \left( \frac{\partial f_1}{\partial y_s}=\cdots=\frac{\partial f_r}{\partial y_s}=0 \right) \subseteq \QSing(X),\]
from which the lemma immediately follows. Denote the left hand side by $W$ and let $y\in W$. Then as $Y$ is not quasi-smooth at $y$, the rank of the Jacobian
\[\left( \frac{\partial f_i}{\partial x_j} \right)_{1\le i\le r, 1\le j\le m+s-1}\]
is less than $r$ (where $x_{m+q}=y_q$, $q=1,\cdots,s$). Since $\frac{\partial f_1}{\partial y_s}=\cdots=\frac{\partial f_r}{\partial y_s}=0$ at $y$, this rank is equal to the rank of the Jacobian for $X$ at $y$. Hence $X$ is not quasi-smooth at $y$ either and $y\in\QSing(X)$ as desired.
\end{proof}

\begin{lem} \label{lem:mult-bound-quasi-smooth}
Let $X\subseteq \bP$ be a quasi-smooth weighted complete intersection of codimension $r$ and let $Z\subseteq X$ be an effective cycle of pure codimension $k$ such that
\[Z \equiv c_1(\cO_X(1))^k\cap [X].\]
Then for every subvariety $S\subseteq X$ of dimension $\ge s+ (2^sk+2^s-1)r$, we have $\mult_S Z\le 1$.
\end{lem}

\begin{proof}
We prove by induction on $s$. When $s=0$ the statement follows from \cite[Proposition 5]{Puk-hypersurface} or \cite[Proposition 2.1]{Suzuki-cpi}. Assume that $s\ge 1$ and that the statement has been proven for smaller values of $s$. We may assume that $\dim X\ge s+ (2^sk+2^s-1)r\ge 2k+r+1$, otherwise there is no such $S$. Then by \cite[Proposition 6]{Dimca-betti}, $X$ has Betti number $b_{2k}(X)=1$, hence every irreducible component of $Z$ is numerically proportional to $Z$ and it suffices to show $\mult_S Z\le 1$ under the assumption that $Z$ itself is irreducible. Let $Y=X\cap (y_s=0)$, then by the previous lemma $\delta_Y \le r-1$. Let $d$ be a sufficiently large and divisible integer and let $Y_1$ be a complete intersection in $Y$ cut out by $r$ general weighted hypersurfaces of degree $d$, then $Y_1$ is a quasi-smooth weighted complete intersection of codimension $2r$ in $\bP(1^m,a_1,\cdots,a_{s-1})$. Let $T=S\cap Y_1$, then $\dim T\ge \dim S - (r+1)\ge s-1+(2^{s-1}k+2^{s-1}-1)\cdot 2r$. If $Z\not\subseteq Y$, then $W=(Z\cdot Y_1)$ is a well-defined cycle of pure codimension $k$ in $Y_1$ such that $W\equiv c_1(\cO(1))^k\cap [Y_1]$. By induction hypothesis we have $\mult_T W\le 1$, hence $\mult_S Z\le 1$ in this case. If otherwise $Z\subseteq Y$, then we may view $Z$ as a cycle of codimension $k-1$ in $Y$ and get a well-defined cycle $W=(Z\cdot Y_1)$ of pure codimension $k-1$ in $Y_1$. By \cite[Proposition 6]{Dimca-betti}, we have $b_{2k-2}(Y_1)=1$, hence there exists some $\lambda\in\bQ$ such that $W\equiv \lambda\cdot  c_1(\cO(1))^{k-1}\cap [Y_1]$. Comparing the degrees of both sides we see that $\lambda=a_s^{-1}\le 1$. Therefore, using our induction hypothesis again we obtain $\mult_T W\le 1$ and thus $\mult_S Z\le 1$ as desired. 
\end{proof}

\begin{cor} \label{cor:mult-bound-wt-cpi}
Let $X\subseteq \bP$ be a weighted complete intersection of codimension $r$ and let $Z\subseteq X$ be an effective cycle of pure codimension $k$ such that $Z \sim_\bQ c_1(\cO_X(1))^k\cap [X]$. Then for every subvariety $S\subseteq X$ of dimension $\ge s+\delta_X+1+ (2^sk+2^s-1)(r+\delta_X+1)$, we have $\mult_S Z\le 1$.
\end{cor}

\begin{proof}
Let $d$ be a sufficiently large and divisible integer and let $X_1$ be a complete intersection in $X$ cut out by $\delta_X+1$ general weighted hypersurfaces of degree $d$, then $X_1$ is quasi-smooth and the result follows from Lemma \ref{lem:mult-bound-quasi-smooth} applied to $X_1$. 
\end{proof}

\begin{rem}
The main point of Lemma \ref{lem:mult-bound-quasi-smooth} and Corollary \ref{cor:mult-bound-wt-cpi} is that there exists an integer $N$ depending only on the discrete data $(k,r,s,\delta_X)$ (and most importantly, not on $\dim X$ or $a_1,\cdots,a_s$) such that $\mult_S Z\le 1$ whenever $\dim S\ge N$. Our choice of $N$ above is probably far from optimal (for example, for smooth cyclic covers over a hypersurface, one can just take $N=2k+1$ by a modification of the above argument), but it is sufficient for our need.
\end{rem}

\begin{proof}[Proof of Theorem \ref{thm:weighted cpi}]
The argument is almost identical to the proof of \cite[Theorem 1.2]{Z-cpi}, so we only give a sketch. Let $M\sim_\bQ -K_X$ be a movable boundary. By Lemma \ref{lem:mult-bound-quasi-smooth}, there exists an integer $N$ depending only on $r$ and $s$ such that $\mult_S M\le 1$ and $\mult_S (M^2)\le 1$ for every subvariety $S\subseteq X$ of dimension $\ge N$. In particular, by \cite[3.14.1]{Kol-sing-of-pair} and \cite[Theorem 0.1]{dFEM-mult-and-lct}, $(X,M)$ (resp. $(X,2M)$) has canonical (resp. lc) singularities outside a set of dimension at most $N-1$ in $X$. Let $x\in X$ and let $Y\subseteq X$ be a complete intersection cut out by $N$ general members of the linear system $|H|$ (it is base point free by assumption) containing $x$, then $(Y,2M|_Y)$ is lc outside $x$. Let $L=K_Y+2M\sim (N+1)H$, then $h^0(Y,L)$ is bounded by a (fixed) polynomial $P(m)$ of degree $N+1$ in $m$. Choose sufficiently large $M>0$ such that $p(m)\le\frac{c^c}{c!}$ (where $c=\dim Y=m+s-1-r-N$) for all $m\ge M$, then by \cite[Corollary 1.8]{Z-cpi}, $(Y,M|_Y)$ is lc at $x$, hence by inversion of adjunction, $(X,M)$ has canonical singularities and $X$ is birationally superrigid by \cite[Theorem 1.4.1]{Che-survey}. Similarly for all effective divisor $D\sim_\bQ -K_X$, we have $\lct(X;D)>\frac{1}{2}$, thus $\alpha(X)>\frac{1}{2}$ and by \cite[Theorem 1.2]{SZ-rigid-imply-stable}, $X$ is also K-stable.
\end{proof}

\section{Hypersurface with ordinary singularities} \label{sec:hypersurface}

In this section we prove Theorem \ref{thm:hypersurface}. The proof is a bit lengthy, so we divide it into several steps. Throughout the section, $X$ always denotes a hypersurface of degree $n+1$ and dimension $n\ge 250$ with only isolated ordinary singularities. We first treat the (super)rigidity of $X$.

\begin{lem} \label{lem:canonical on blowup}
Let $x\in X$ be a point of multiplicity $m_0\ge n-4$ and let $\pi:Y\rightarrow X$ be the blowup of $X$ at $x$ with exceptional divisor $E$. Let $M\sim_\bQ -K_X$ be a movable boundary and let $\tM$ be its strict transform on $Y$. Then $(Y,\tM)$ is canonical along $E$.
\end{lem}

\begin{proof}
We start with some multiplicity estimate. By assumption, $E\subseteq \bP^n$ is a smooth hypersurface of degree $m_0$. Let $Z$ be the codimension $2$ cycle $\tM^2$ in $Y$. Note that by \cite[Proposition 5]{Puk-hypersurface}, we have $\mult_y M^2\le 1$ outside a finite union of surfaces, hence $\mult_y Z\le 1$ away from $E$ and another set of dimension at most $2$. Decompose $Z$ into $Z=Z_1+Z_2$ such that the irreducible components of $Z_1$ (resp. $Z_2$) is contained (resp. not contained) in $E$. We may view $Z_1$ as a divisor in $E$ and since $E$ has Picard number one, there exists some $b>0$ such that $Z_1\sim_\bQ b\cdot c_1(\cO_E(1))\cap [E]\sim_\bQ -bE^2$. By \cite[Proposition 5]{Puk-hypersurface} (note that $E$ is a smooth hypersurface), we have $\mult_y Z_1\le b$ outside a finite number of points. Let $c=\ord_E M$, then we have $\tM\sim_\bQ \pi^*M-cE$, $Z\sim_\bQ \pi^*H^2+c^2E^2$ and $Z_2=Z-Z_1\sim_\bQ \pi^*H^2+(c^2+b)E^2$. Let $W=Z_2\cdot E$, then $W\sim_\bQ (c^2+b)E^3$ and as a codimension $2$ cycle in $E$ we have $W\sim_\bQ (c^2+b) c_1(\cO_E(1))^2\cap [E]$, hence by \cite[Proposition 5]{Puk-hypersurface} again we see that $\mult_y W\le c^2+b$ outside a finite union of curves in $E$. On the other hand, as $\pi^*H-E$ is a nef line bundle on $Y$, we have
\[b m_0=Z_1\cdot (\pi^*H-E)^{n-2}\le Z\cdot (\pi^*H-E)^{n-2}=\tM^2\cdot (\pi^*H-E)^{n-2}=n+1-c^2 m_0,\]
hence $b\le c^2+b\le \frac{n+1}{m_0}$ and as $m_0\ge n-4$ and $n\ge 250$ by assumption, we obtain $\mult_y Z\le (c^2+2b)\le \frac{2(n+1)}{m_0}<\frac{9}{4}$ outside a subset $V$ of dimension at most $2$ in $Y$ (at least in a neighbourhood of $E$; the same remark applies to the other claims about singularities in this proof). In particular, if $y\in Y\backslash V$ and $S$ is a general surface section in $Y$ containing $y$, then by \cite[Theorem 3.1]{Corti-4n^2-ineq} (or \cite[Theorem 0.1]{dFEM-mult-and-lct}), $(S,\frac{4}{3}\tM|_S)$ is klt at $y$ and thus by inversion of adjunction, $(Y,\frac{4}{3}\tM)$ is klt away from $V$.

We also need to show that $(Y,\tM)$ has canonical singularities away from $V$. Let $y\in E\backslash V$ and let $K\subseteq E$ be a center of maximal singularity containing $y$. If $K$ has codimension at least $3$ in $Y$ then by adjunction $(S,\tM|_S)$ is not lc at $y$ where $S$ is a general surface section in $Y$ containing $y$, but from the above discussion $(S,\frac{4}{3}\tM|_S)$ is klt at $y$, a contradiction. Thus $K$ has codimension $2$ in $Y$ and is a divisor in $E$ (the pair $(X,M)$ is already canonical outside a finite union of curves, see the proof of \cite[Theorem 1.2]{LZ-singular-cpi}) and we have $\mult_K \tM>1$. It follows that (in the above notations) $b>1$ and $c>1$, but $c^2+b\le \frac{n+1}{m_0}<2$, a contradiction. Thus $(Y,\tM)$ has canonical singularities away from $V$.

The rest of the proof is similar to those in \cite{Z-cpi}. Let $y\in E$, let $Y'$ be a complete intersection in $Y$ cut out by three general hypersurfaces in $|\pi^*H-E|$ containing $y$ and let $M'=\tM|_{Y'}$. Then $(Y',M')$ is klt away from $y$ and it is not hard to verify that for $L=5\pi^*H+(n-3-m_0)E$, $L-(K_{Y'}+\frac{4}{3}M')$ is nef and big. By \cite[Theorem 3.3]{Z-cpi} and Lemma \ref{lem:nlc estimate} (applied to $\lambda=3$), we see that $(Y',M')$ is lc as long as
\begin{equation} \label{ineq:h^0(L)}
    h^0(Y',L)\le 2^{\frac{n-3}{3}-1}.
\end{equation}
But since $\pi(Y')$ is a complete intersection in $\bP^{n-2}$, we have
\[h^0(Y',L)\le h^0(\bP^{n-2},\cO_{\bP^{n-2}}(5))\le \binom{n+3}{5}\]
and then it is not hard to see that \eqref{ineq:h^0(L)} holds when $n\ge 250$. Therefore $(Y',M')$ is lc and since $\dim V\le 2$, $(Y,\tM)$ is canonical by \cite[Lemma 3.10]{Z-cpi}.
\end{proof}

The following lemma is used in the above proof.

\begin{lem} \label{lem:nlc estimate}
Let $X$ be a smooth variety of dimension $n$ and $x\in X$. Let $\lambda>0$, then $\lnlc(x,X;\lambda) > 2^{\frac{n}{\lambda}-1}$.
\end{lem}

\begin{proof}
This follows directly from \cite[Lemma 3.4]{Z-cpi} and the proof of \cite[Lemma A.5]{Z-cpi}.
\end{proof}

We are ready to prove the birational (super)rigidity part of Theorem \ref{thm:hypersurface}.

\begin{proof}[Proof of Theorem \ref{thm:hypersurface} (2)(3)]
Let $M\sim_\bQ -K_X$ be a movable boundary. Suppose that $(X,M)$ is not canonical at some point $x\in X$ with multiplicity $m_0$. By (the proof of) \cite[Theorem 1.2]{LZ-singular-cpi}, $(X,M)$ has canonical singularities outside points of multiplicity at most $n-5$. In particular, $(X,M)$ is canonical in a punctured neighbourhood of $x$ and we have $m_0\ge n-4$. Let $\pi:Y\rightarrow X$ be the blowup of $x$, let $\tM$ be the strict transform of $M$ and let $E$ be the exceptional divisor. By Lemma \ref{lem:canonical on blowup}, $(Y,\tM)$ has canonical singularities, thus as $(X,M)$ is not canonical, $E$ must be a center of maximal singularity and we obtain $a(E;X,M)=n-m_0-\ord_E M<0$. As $M$ is movable and $\pi^*H-E$ is nef (where $H$ is the hyperplane class on $X$), we have $(\tM^2\cdot (\pi^*H-E)^{n-2})\ge 0$ and hence $(n-m_0)^2 m_0 < (\ord_E M)^2 m_0 \le \deg X=n+1$, which can only be true when $m_0=n-1$. It follows that the only possible maximal singularity of $X$ is the ordinary blowup of a point of multiplicity $n-1$. By standard argument (see e.g. \cite[\S 3]{CPR}), this proves parts (2) and (3) of Theorem \ref{thm:hypersurface}.
\end{proof}

Next we show that $X$ is K-stable when all points have multiplicity at most $n-1$. By \cite[Theorem 1.2]{SZ-rigid-imply-stable}, it suffices to show that $\alpha(X)>\frac{1}{2}$ and that in case $X$ has multiplicity $n-1$ at some point $x$, $(X,M)$ is lc at $x$ for every movable boundary $M\sim_\bQ -K_X$ (using a modification of the above argument).

\begin{lem} \label{lem:K-stable mult<=n-1}
Let $x\in X$ be a point of multiplicity $m_0\le n-1$, then $(X,M)$ is lc at $x$ for every movable boundary $M\sim_\bQ -K_X$ and $(X,\frac{1}{2}D)$ is klt at $x$ for every effective divisor $D\sim_\bQ -K_X$.
\end{lem}

\begin{proof}
As before let $\pi:Y\rightarrow X$ be the blowup of of $x$ with exceptional divisor $E$ and let $\tM$ be the strict transform of $M$. From the above proof of Lemma \ref{lem:canonical on blowup} and Theorem \ref{thm:hypersurface}(2), we see that $(X,M)$ is indeed canonical at $x$ if $m_0\le n-2$. Hence we may assume that $m_0=n-1$. We have $K_Y+\tM+(c-1)E=\pi^*(K_X+M)$ where $c=\ord_E M$ and as before since $\tM^2\cdot (\pi^*H-E)^{n-2}\ge 0$, we obtain $c^2\le \frac{n+1}{m_0}=1+\frac{2}{n-1}$ (hence $c-1<\frac{1}{n-1}$) and $\mult_y \tM^2<\frac{2(n+1)}{n-1}$ outside a set $V$ of dimension at most $2$. We may assume that $c>1$, otherwise $(X,M)$ is already canonical by Lemma \ref{lem:canonical on blowup}. By \cite[Theorem 3.1]{Corti-4n^2-ineq} applied to a general surface section of $(Y,(c-1)E)$ containing some $y\in Y\backslash V$ as before (noting that $\frac{16}{9}\mult_y \tM^2\le 4(1-\frac{4}{3}(1-c))$), we find that $(Y,\frac{4}{3}\Gamma)$ is lc away from $V$ (where $\Gamma=(c-1)E+\tM$) and the same argument as before (i.e. using \cite[Theorem 3.3]{Z-cpi} and Lemma \ref{lem:nlc estimate}) implies that $\lct(Y;\frac{4}{3}\Gamma)\ge\frac{3}{4}$ along $E$. In particular, $(Y,\tM+(c-1)E)$ is lc along $E$ and $(X,M)$ is lc at $x$ as desired.

Similarly, let $\tD$ be the strict transform of $D$ on $Y$, then $\tD\sim_\bQ \pi^*H-aE$ for some $a\le \frac{n+1}{m_0}$ (as $\tD\cdot(\pi^*H-E)^{n-1}\ge 0$) and $\mult_y \tD\le \frac{n+1}{m_0}$ outside a set $V_1$ of dimension at most $1$ (by bounding the degree of $\tD|_E$ and using \cite[Proposition 5]{Puk-hypersurface}). If $m_0\le n-5$ then by the proof of \cite[Theorem 1.2]{LZ-singular-cpi}, $(X,\frac{1}{2}D)$ is klt at $x$. If $m_0\ge n-4$, then by the above multiplicity estimate $(Y,\lambda\tD)$ is klt away from $V_1$ for some $\lambda>\frac{3}{4}$ and a similar argument as before involving \cite[Theorem 3.3]{Z-cpi} and Lemma \ref{lem:nlc estimate} (this time applied to $\lambda=2$) implies that $\lct(Y;\lambda\tD) \ge \frac{2}{3}$ and thus $(Y,\frac{1}{2}\tD)$ is klt along $E$. Since $a(E;X,\frac{1}{2}D) = n-m_0-\frac{1}{2}a\ge 1-\frac{1}{2}a>0$, $(X,\frac{1}{2}D)$ is also klt at $x$.
\end{proof}

Finally we treat the case when some points of $X$ has multiplicity $n$. For this we use the criterion \cite[Theorem 1.5]{Z-cpi} and need to analyze the singularities of a few more auxiliary pairs.

\begin{lem} \label{lem:D/n+(1-1/n)M lc}
Let $D\sim_\bQ -K_X$ be an effective divisor and $M\sim_\bQ -K_X$ a movable boundary. Then $(X,\frac{1}{n}D+\frac{n-1}{n}M)$ is lc over the smooth locus of $X$.
\end{lem}

\begin{proof}
Let $\Gamma=\frac{1}{n}D+\frac{n-1}{n}M$ and let $c=\frac{3}{2}$. We first show that $(X,c\Gamma)$ is lc outside a set of dimension at most $2$. By \cite[Proposition 5]{Puk-hypersurface}, we have $\mult_x D\le 1$ and $\mult_x (M^2)\le 1$ away from a set $Z$ of dimension at most $2$. Let $S\subseteq X$ be a general surface section containing a point $x\in X\backslash Z$, then it's not hard to verify that
\[\frac{4c}{n}\mult_x (D|_S)+c^2\left( \frac{n-1}{n} \right)^2\mult_x (M|_S^2) \le \frac{4c}{n}+c^2\left( 1-\frac{1}{n} \right)^2 < 4,\]
hence by the following Lemma \ref{lem:surface D+M}, $(S,c\Gamma|_S)$ is lc at $x$ and therefore by inversion of adjunction we see that $(X,c\Gamma)$ is lc away from $Z$. Now let $y\in X$ be any smooth point and let $Y\subseteq X$ be a general linear section of codimension $2$ containing $y$, then $(Y,c\Gamma|_Y)$ is lc in a punctured neighbourhood of $y$ and $L-(K_Y+c\Gamma|_Y)$ is ample where $L=3H|_Y$. By \cite[Theorem 3.3]{Z-cpi} and Lemma \ref{lem:nlc estimate} (applied to $\lambda=2$), $\lct(Y,c\Gamma|_Y)\ge \frac{2}{3}=c^{-1}$ (in a neighbourhood of $y$) as long as $h^0(Y,L)=\binom{n+2}{3}\le 2^{\frac{n-2}{2}-1}$, which holds when $n\ge 250$. Therefore, $(Y,\Gamma|_Y)$ is lc at $y$ and by inversion of adjunction $(X,\Gamma)$ is also lc at $y$.
\end{proof}

\begin{lem} \label{lem:surface D+M}
Let $S$ be a smooth surface and let $x\in S$. Let $D$ $($resp. $M)$ be an effective divisor $($resp. a movable boundary$)$ on $S$. Assume that $4\mult_x D + \mult_x (M^2)\le 4$. Then $(S,D+M)$ is lc at $x$.
\end{lem}

\begin{proof}
Let $m$ be a sufficiently divisible integer such that $mD$ and $mM$ are both integral. Let $\mu$ be such that $\mu^{-1}=\lct_x(S,m(D+M))$. By \cite[Theorem 2.2]{dFEM-mult-and-lct}, we have $4\mu\cdot \mult_x(mD)+\mult_x(m^2M^2)\ge 4\mu^2$. Equivalently,
\[
\frac{4m}{\mu}\mult_x(D)+\frac{m^2}{\mu^2}\mult_x(M^2)\ge 4.
\]
By assumption, this implies $\mu\le m$. In other words, $(S,D+M)$ is lc at $x$.
\end{proof}

\begin{lem} \label{lem:D_0/n lc}
Let $V\subseteq \bP^n$ be a smooth Fano hypersurface and let $D\sim_\bQ \ell H$ be an effective divisor on $V$ for some $\ell\le n(n-2)$ (where $H$ is the hyperplane class). Assume that $(V,D)$ is lc away from a finite number of points, then $(X,\frac{1}{n}D)$ is lc. 
\end{lem}

\begin{proof}
By \cite[Theorem 3.3]{Z-cpi} and Lemma \ref{lem:nlc estimate} (applied to $L=n(n-2)H$ and $\lambda=\frac{1}{n-1}$ respectively), we have $\lct(X;D)\ge \frac{1}{n}$ as long as $h^0(X,L)=\binom{n+n(n-2)}{n}\le 2^{n(n-1)-1}$. Since $\binom{n+m}{m}\le \frac{1}{2}2^{n+m}$ for all $n\neq m$, the result follows.
\end{proof}

We are now ready to finish the proof of Theorem \ref{thm:hypersurface}.

\begin{proof}[Proof of Theorem \ref{thm:hypersurface}(1)]
Let $D\sim_\bQ -K_X$ be an effective divisor on $X$ and let $M\sim_\bQ -K_X$ be a movable boundary. We claim that $(X,\Delta=\frac{n}{n+1}\Gamma)$ is klt where $\Gamma=\frac{1}{n}D+\frac{n-1}{n}M$. It then follows from \cite[Theorem 1.5]{Z-cpi} that $X$ is K-stable.

As $X$ has Picard number one, we may assume that $D$ is irreducible (being klt is preserved under convex linear combination). Let $x\in X$ be a point of multiplicity $m_0$. If $m_0\le n-1$, then by Lemma \ref{lem:K-stable mult<=n-1}, $(X,M)$ is lc at $x$ and $(X,\frac{1}{2}D)$ is klt at $x$, hence as $\Delta=\frac{2}{n+1}\cdot \frac{1}{2}D+(1-\frac{2}{n+1})M$ is a convex linear combination of $\frac{1}{2}D$ and $M$, we see that $(X,\Delta)$ is klt at $x$. It remains to consider the case $m_0=n$. After a change of coordinate, we may assume that $x=[0:\cdots:0:1]$ and the defining equation of $X$ can be written as $x_{n+1}f_n(x_0,\cdots,x_n)+f_{n+1}(x_0,\cdots,x_n)=0$ where $f_i$ ($i=n,n+1$) is homogeneous of degree $i$. Let $D_0=(f_n=0)\cap X$. As before, let $\pi:Y\rightarrow X$ be the blowup of $x$ with exceptional divisor $E$ and strict transforms $\tD$, $\tM$, $\tD_0$, $\tGamma$ and $\tDelta$.

Suppose first that $D$ is not supported on $D_0$, then as before we have $\ord_E D\le 1$ since $\tD\cdot \tD_0\cdot (\pi^*H-E)^{n-2}\ge 0$. Similarly $\ord_E M\le 1$ and thus $\mu=\ord_E\Gamma\le 1$. We claim that in a neighbourhood of $E$, $(Y,\pi^*\Gamma)$ is lc outside a finite number of point. By Lemma \ref{lem:D/n+(1-1/n)M lc}, $(X,\Gamma)$ is lc in a punctured neighbourhood of $x$. Thus as $E$ appears with coefficient at most $1$ in $\pi^*\Gamma$, by inversion of adjunction it suffices to show that $(E,\tGamma|_E)$ is lc outside a finite number of points. But as $E$ is a smooth hypersurface in $\bP^n$ and $\tGamma|_E\sim_\bQ -\mu E|_E\sim_\bQ \mu c_1(\cO_E(1))\cap [E]$ where $\mu\le 1$, this follows from \cite[Proposition 5]{Puk-hypersurface} and therefore $(Y,\pi^*\Gamma)$ is lc outside a finite number of point. By \cite[Theorem 1.6]{Z-cpi} applied to $(Y,\pi^*\Gamma)$ with $L=0$ (note that $K_Y+\pi^*\Gamma\sim_\bQ \pi^*(K_X+\Gamma)\sim_\bQ 0$ and $\pi^*\Gamma$ is nef and big) we obtain $\lct(Y;\pi^*\Gamma)\ge\frac{n}{n+1}$ with equality if and only if $\mult_y (\pi^*\Gamma) = n+1$ for some $y\in Y$. But if such $y$ exists, then clearly $y\in E$ and since $\pi^*\Gamma=\tGamma+\mu E$ where $\mu\le 1$, we have $\mult_y(\tGamma|_E)\ge\mult_y\tGamma\ge n$. Recall that $E$ is a smooth hypersurface of degree $n$ and $\tGamma|_E\sim_\bQ \mu c_1(\cO_E(1))$, we also have $\mult_y(\tGamma|_E)\le n\mu\le n$. Hence we must have equality everywhere, in particular, $\mu=1$ and $\mult_y\tM=\mult_y(\tM|_E)=n$. But as $\tM$ is movable, if we choose $n-3$ general members $P_1,\cdots,P_{n-3}$ of $|\pi^*H-E|$ passing through $y$ and another general member $Q$ of $|2\pi^*H-E|$ (which is very ample) containing $y$, then the intersection $\tM^2\cdot P_1\cdot\ldots\cdot P_{n-3}\cdot Q$ is zero dimensional and we get
\[n+2=2(H^n)-\deg E=(\tM^2\cdot P_1\cdot\ldots\cdot P_{n-3}\cdot Q)\ge (\mult_y\tM)^2=n^2,\]
a contradiction. Thus such $y$ doesn't exist and we indeed have $\lct(Y;\pi^*\Gamma)>\frac{n}{n+1}$. In other words, $(Y,\pi^*\Delta)$ (and hence also $(X,\Delta)$) is klt.

It remain to treat the case when $D$ is supported on $D_0$ (i.e. $D=\frac{1}{n}D_0$). As $\ord_E M\le 1$ and $\ord_E D=\frac{n+1}{n}$, we still have $\ord_E\Delta < 1$. Since $X=(x_{n+1}f_n+f_{n+1}=0)$ has only isolated singularities, a direct computation shows that $\tD_0|_E=(f_{n+1}=0)\cap E$ has only isolated singularities as well. By Lemma \ref{lem:D_0/n lc}, $(E,\tD|_E=\frac{1}{n}\tD_0|_E)$ is lc. On the other hand, as $E$ is a smooth hypersurface of degree $n$ in $\bP^n$ and $\tM|_E\sim_\bQ \mu c_1(\cO_E(1))$ where $\mu\le 1$, we see that $(E,\frac{n-1}{n}\tM|_E)$ is also lc by \cite[Corollary 1.7]{Z-cpi}. Taking convex linear combination $\tDelta=\frac{1}{n+1}\tD|_E+\frac{n}{n+1}\cdot \frac{n-1}{n}\tM|_E$, it follows that $(E,\tDelta|_E)$ is lc and thus by inversion of adjunction $(Y,E+\tDelta)$ is lc as well. As $(X,\Delta)$ is klt away from $x$, all the lc centers of $(Y,E+\tDelta)$ are contained in $E$, hence as $\ord_E\Delta < 1$, we deduce that $(Y,\pi^*\Delta)$ and $(X,\Delta)$ are both klt and this finishes the proof.
\end{proof}

\section{Counterexample} \label{sec:counterexample}

In this section we complete the proof of Theorem \ref{thm:counterexample}.

\begin{proof}[Proof of Theorem \ref{thm:counterexample}]
It is easy to see that the $\cX_{0,t}$ ($t\neq0$) is not birationally superrigid since the assumptions imply that $\mult_x\cX_{0,t}=2m-2$, hence a general line through $x$ in $\bP^{n+1}$ intersects the hypersurface $\cX_{0,t}$ in exactly two other points and therefore induces a birational involution by interchanging these two point. By \cite{dF-hypersurface} (applied to the smooth hypersurfaces $\cX_{s,t}$ when $st\neq 0$) and Theorem \ref{thm:weighted cpi} (applied to $\cX_{s,0}$, which are smooth double covers of hypersurfaces when $s\neq 0$), it is clear that $\cX_{s,t}$ is birationally superrigid when $s\neq 0$. It remains to show that $X=\cX_\origin$ is birationally superrigid where $\origin=(0,0)\in\bA^2$. For ease of notation, we let $f=f_0$, $F=F_0$, etc.

We may identify $x\in F\cap G$ with its preimage in $X$. By assumption $X$ is smooth outside $x$ and $-K_X\sim H$ where $H$ is the pullback of the hyperplane class on $G$. Let $M\sim_\bQ -K_X$, by \cite[Theorem 1.4.1]{Che-survey}, it suffices to show that $(X,M)$ has canonical singularities. Note that $X=(y^2-f=g=0)\subseteq \bP(1^{n+2},m)$ is a weighted complete intersection, thus by Corollary \ref{cor:mult-bound-wt-cpi}, there exists a constant (i.e. independent of $m$) $N_0\in \bZ$ such that $\mult_y (M^2)\le 1$ away from a subset of dimension at most $N_0$. It then follows from the proof of Theorem \ref{thm:weighted cpi} that there exists another constant $N_1$ such that $(X,M)$ is canonical over the smooth locus of $X$ when $n\ge N_1$. It remains to show that the pair is also canonical at $x$. Let $x_0,\cdots,x_{n+1},y$ be the weighted homogeneous coordinate of $\bP(1^{n+2},m)$ and after a change of coordinates we may put $x=[0:\cdots:0:1:0]$. Let $\pi:Y\rightarrow X$ be the weighted blowup at $x$ with associated weights $\wt(x_i)=1$ ($i=0,\cdots,n$) and $\wt(y)=m-1$. Then the exceptional divisor $E$ is isomorphic to the double cover of the projective tangent cone of $G$ branched over the projective tangent cone of $F$. In particular, $E\subseteq\bP(1^{n+1},m-1)$ is a smooth weighted complete intersection and $Y$ is also smooth. We can write $g=x_{n+1}g_{m-1}(x_0,\cdots,x_n)+g_m(x_0,\cdots,x_n)$ where $\deg(g_i)=i$ ($i=m-1,m$). Let $D=(g_{m-1}=0)\subseteq X$, then $\ord_E D=m$. Let $c=\ord_E M$ and let $\tM\sim_\bQ \pi^*H-cE$ (resp. $\tD\sim (m-1)\pi^*H-mE$) be the strict transform of $M$ (resp. $D$) on $Y$. Since $\tM$ is movable and $\pi^*H-E$ is base point free on $Y$, we have $0\le \tM\cdot \tD\cdot (\pi^*H-E)^{n-2}=(m-1)(H^n)-cm\deg E=2m(m-1)(1-c)$, thus $c\le 1$. We now show that $(Y,\tM)$ is canonical, then as $K_Y+\tM=\pi^*(K_X+M)+(1-c)E$, $(X,M)$ is also canonical as desired. The rest of the argument is similar to Theorem \ref{thm:hypersurface}. Let $Z=\tM^2$ and write $Z=Z_1+Z_2$ such that the irreducible components of $Z_1$ (resp. $Z_2$) are (resp. not) contained in $E$. We have $Z_1\sim_\bQ -bE^2$ for some $b\ge 0$. As in the proof of Theorem \ref{thm:hypersurface}, we have $(b+c^2)\deg E\le (H^n)$ (or equivalently $b+c^2\le \frac{m}{m-1}$. It then follows from Lemma \ref{lem:mult-bound-quasi-smooth} that $\mult_y Z_i\le \frac{m}{m-1}$ and $\mult_y Z\le \frac{2m}{m-1}$ outside a set of dimension at most $N_0$ (possibly by increasing the constant $N_0$). By \cite[Theorem 3.1]{Corti-4n^2-ineq} or \cite[Theorem 0.1]{dFEM-mult-and-lct}, there exists some $\mu>1$ such that $(Y,\mu\tM)$ is lc outside a set of dimension at most $N_0$. The same proof as in Theorem \ref{thm:hypersurface} then implies that $(Y,\tM)$ is canonical when $m\gg 0$. The proof is now complete.
\end{proof}


We conclude the section with some questions. First, we expect some weaker variant of birational rigidity to be an open property. To make this precise we need a definition.

\begin{defn} \cite[Definition 1.4]{AO-pfaffian}
A Fano variety $X$ is said to be \emph{birationally solid} if there is no birational map $X\dashrightarrow Y$ to the source of a strict (i.e. the base has dimension at least $1$) Mori fiber space.
\end{defn}

Note that in our main example (Theorem \ref{thm:counterexample}) and in the 3-dimensional example \cite[Example 6.3]{CP-3fold-hypersurface}, every member of the family is birationally rigid and in particular solid. In the example of \cite{CG-rigidity-not-open}, it is also proved that a general member of the family is solid since there is at most one other Mori fiber space (which belongs to the same family) that's birational to it. These motivate the following question.

\begin{que}
Let $f:\cX\rightarrow T$ be a $\bQ$-Gorenstein family of Fano varieties. Is the set of $t\in T$ such that $\cX_t$ is birationally solid an open subset of $T$?
\end{que}

In another direction, recall that birational superrigidity is a constructible condition \cite[Corollary 7.8]{SC-log-model}; since birational superrigidity is closely related to the higher codimensional alpha invariants, we may ask:

\begin{que} \label{que:constructible}
Let $f:\cX\rightarrow T$ be a $\bQ$-Gorenstein family of Fano varieties. Is the function $t\mapsto \alpha^{(k)}(\cX_t)$ (or its analogue with canonical in place of log canonical in the definition) on $T$ a constructible function for each $k$?
\end{que}

\appendix

\section{On a conjecture of Tian}

In this appendix, we study some special cases of Question \ref{que:constructible} for the usual alpha invariants in a slightly more general context. In \cite[Conjecture 5.3]{Tian-alpha-defn}, Tian conjectured that in the definition \eqref{eq:alpha^(k)} of alpha invariants, the infimum is indeed a minimum:

\begin{conj} \label{conj:Tian}
Let $(X,\Delta)$ be a klt pair and $L$ an ample line bundle on $X$. Then there exists an integer $m>0$ such that
\[\alpha(X,\Delta;L)=\alpha_m(X,\Delta;L).\]
\end{conj}

This is confirmed by Birkar \cite[Theorem 1.5]{Birkar-2} when $(X,\Delta)$ is log Fano, $L=-(K_X+\Delta)$ and $\alpha(X,\Delta;L)\le 1$. Another special case is smooth quartic surfaces \cite[Theorem 1.2]{on-conj-of-tian} where one can just take $m=1$. The main result of this appendix is:

\begin{thm} \label{thm:Tian's conj}
Fix $r\in\bN$ and $\epsilon>0$. Then there exists an integer $N$ such that for every smooth complete intersection $X\subseteq \bP^{n+r}$ of codimension $r$ and dimension $n\ge N$ such that $K_X\sim sH$ where $s\le (1-\epsilon)n$ and $H$ is the hyperplane class, we have $\alpha(X;H)=\alpha_m(X;H)$ for some integer $m$ that only depends on $n,r$ and $s$.
\end{thm}

Note that in the context of Question \ref{que:constructible}, if Conjecture \ref{conj:Tian} holds for the fibers of $f$ with a uniform choice of $m$, then the alpha invariant function is constructible. Hence the above result implies the following special case of Question \ref{que:constructible}:

\begin{cor} \label{cor:constructible}
Let $r\in\bN$ and let $\cX\subseteq \bP^{n+r}\times T$ be a smooth family of Fano complete intersections of codimension $r$ over $T$. Assume that $n\ge 10r$, then the function $t\mapsto \alpha(\cX_t;H)$ is constructible.
\end{cor}

The proof of Theorem \ref{thm:Tian's conj} is based on the following criterion.

\begin{prop} \label{prop:alpha^2>alpha}
Let $(X,\Delta)$ be a klt pair of dimension $n$ and $L$ an ample $\bQ$-Cartier divisor on $X$. Let $r>0$, $s\ge 0$ be integers. Assume that
    \begin{enumerate}
        \item $sL-(K_X+\Delta)$ is nef and $rL$ is globally generated,
        \item the class group $\mathrm{Cl}(X)$ is generated by $L$,
        \item $\alpha :=\alpha(X,\Delta;L)<\alpha_2 :=\alpha^{(2)}(X,\Delta;L)$.
    \end{enumerate}
Then either $\alpha(X,\Delta;L)=\lct(X,\Delta;D)$ for some $D\sim_\bQ L$ supported on an irreducible component of $\Delta$ or there exists an integer $m\le \frac{(s+nr+\alpha_2)\alpha}{\alpha_2 -\alpha}$ such that $\alpha(X,\Delta;L)=\alpha_m(X,\Delta;L)$.
\end{prop}

\begin{proof}
The argument is very similar to the proof of Propsition \ref{prop:uniform convergence of alpha}. Let $D\sim_\bQ L$ be an effective divisor. Since $X$ is $\bQ$-factorial and $\rho(X)=1$ by assumption, each irreducible component of $D$ is $\bQ$-linearly equivalent to some rational multiple of $L$. As being lc is closed under convex linear combination, we may replace $D$ by the suitable multiple of one of its irreducible components without increasing $\lct(X,\Delta;D)$. It follows that $\alpha(X,\Delta;L)$ is also the infimum of $\lct(X,\Delta;D)$ for all irreducible $D\sim_\bQ L$.

Let $0<\epsilon\ll 1$. Let $D\sim _\bQ L$ be an irreducible divisor such that $\lct(X,\Delta;D)<\alpha+\epsilon$. Write $D=\lambda\Gamma$ for some irreducible and reduced divisor $\Gamma\subseteq X$ and $\lambda>0$. Since $\mathrm{Cl}(X)$ is generated by $L$, we have $\Gamma\in |mL|$ for some integer $m$ and thus $\lambda=\frac{1}{m}$ since $D\sim _\bQ L$. As in the proof of Propsition \ref{prop:uniform convergence of alpha}, for sufficiently small $c>0$ we have
\[\ord_E \cJ(X,\Delta+m(1-c)D) > A \left( \frac{m(1-c)}{\alpha+\epsilon}-1 \right)\]
where $E$ is an exceptional divisor over $X$ that computes $\lct(X,\Delta;D)$ and $A=a(E;X,\Delta)+1$ is the log discrepancy.
Letting $c\rightarrow 0$ we get
\[\ord_E \cJ(X,\Delta+m(1-c)D) \ge A \left( \frac{m}{\alpha+\epsilon}-1 \right).\]
Moreover, $\cJ(X,\Delta+m(1-c)D)\otimes \cO_X((m+s+nr)L)$ is globally generated and gives rise to a sub linear series $\cM\subseteq |(m+s+nr)L|$ with
\begin{equation} \label{eq:lct(M)}
    \lct(X,\Delta;\frac{1}{m+s+nr}M)\le \frac{(m+s+nr)A}{\ord_E \cJ(X,\Delta+m(1-c)D)}\le\frac{m+s+nr}{m-\alpha-\epsilon}(\alpha+\epsilon).
\end{equation}
Now if $\Gamma$ is not a component of $\Delta$, then $(X,\Delta+m(1-c)D)$ is klt in codimension one, thus the base locus of $\cM$ has codimension at least two and we have $\lct(X,\Delta;\frac{1}{m+s+nr}M)\ge \alpha_2$ by definition. Combined with \eqref{eq:lct(M)} this yields
\[\alpha_2\le \frac{m+s+nr}{m-\alpha-\epsilon}(\alpha+\epsilon),\]
or equivalently (using $0<\epsilon\ll 1$), $m\le \frac{(s+nr+\alpha_2)\alpha}{\alpha_2 -\alpha}$. In other words, we have shown that $\alpha(X,\Delta;D)$ is either computed by a component of $\Delta$ or given by the infimum of $\lct(X,\Delta;\frac{1}{m}\Gamma)$ over all $m\le \frac{(s+nr+\alpha_2)\alpha}{\alpha_2 -\alpha}$ and $\Gamma\in |mL|$. This concludes the proof.
\end{proof}

Thus in order to prove Theorem \ref{thm:Tian's conj}, we need to exhibit a gap between the codimension 2 alpha invariant and the usual alpha invariant of a smooth complete intersection. Since it is clear that $\alpha(X;H)\le 1$, it suffices to show that $\alpha^{(2)}(X;H)>1+\delta$ for some absolute constant $\delta>0$. This can be done using the same argument as in \cite{Z-cpi}.

\begin{lem} \label{lem:lct>1/2+delta}
There exists a constant $\delta=\delta(n)>0$ depending only on $n$ such that if $X$ is a smooth projective variety of dimension $n$, $D$ an effective $\bQ$-divisor on $X$ and $L$ a line bundle such that
    \begin{enumerate}
        \item $L-(K_X+(1-\epsilon)D)$ is nef and big for all $0<\epsilon\ll 1$,
        \item $(X,D)$ is lc outside a finite number of points, and
        \item $h^0(X,L)<\frac{1}{n}\binom{2n-2}{n-1}$,
    \end{enumerate}
then $\lct(X;D)>\frac{1}{2}+\delta$.
\end{lem}

\begin{proof}
Let $\bar{\sigma}_n = \min\{\#(Q_\mathbf{a}\cap \bZ^n)\,|\,\mathbf{a}\in\bR^n_+\;\mathrm{s.t.}\;(1,1,\cdots,1)\in\overline{Q_\mathbf{a}}\}$ where
\[Q_{\mathbf{a}} = \{\mathbf{x}\in\bR^n_{\ge0}\,|\,\mathbf{a}\cdot\mathbf{x} < 1\}.\]
By the proof of \cite[Lemma 3.4]{Z-cpi} (or \cite[Theorem 1.1]{dFEM-mult-and-lct}), if $\cI$ is a monomial ideal co-supported at $0\in \bA^n$ such that $\ell(\cO_{\bA^n}/\cI)<\bar{\sigma}_n$, then $(\bA^n,\cI)$ is klt. Since there are only finitely many such monomial ideals, there exists a constant $\delta_0>0$ depending only on $n$ such that $\lct(\bA^n;\cI)>1+\delta_0$ for all such $\cI$. Hence by lower semi-continuity of lct as in the proof of \cite[Lemma 3.4]{Z-cpi}, for any $x\in X$ and $\cI\subseteq \cO_X$ co-supported at $x$ such that $\ell(\cO_X/\cI)<\bar{\sigma}_n$, we have $\lct(X;\cI)>1+\delta_0$. Therefore, by \cite[Theorem 1.6]{Z-cpi} we have $\lct(X;D) > \frac{1+\delta_0}{2+\delta_0} > \frac{1}{2}+\delta$ for some constant $\delta$ depending only on $n$ as long as $h^0(X,L)<\bar{\sigma}_n$. Thus it remains to prove that $\bar{\sigma}_n\ge \frac{1}{n}\binom{2n-2}{n-1}$. This comes from the fact that if $(1,1,\cdots,1)\in \overline{Q_\mathbf{a}}$ so that $\sum_{i=1}^n a_i\le 1$ and $m_1,\cdots,m_n\in \bZ$ satisfy $\sum_{i=1}^n m_i\le n-1$, then at least one cyclic permutation of $(m_1,\cdots,m_n)$ lies in $Q_\mathbf{a}$ (see \cite[Paragraph 57]{Kollar-survey}).
\end{proof}

\begin{lem} \label{lem:alpha^2 of cpi}
Fix $r\in \bN$, $\epsilon>0$ and let $N\gg 0$. Let $X\subseteq\bP^{n+r}$ be a smooth complete intersection of dimension $n$ and codimension $r$. Let $H$ be the hyperplane class. Suppose $K_X=sH$, $s\le (1-\epsilon)n$ and $n\ge N$, then $\alpha^{(2)}(X;H)>1+\delta$ for some $\delta>0$ that only depends on $n$ and $r$.
\end{lem}

\begin{proof}
Let $M\sim_\bQ H$ be a movable boundary on $X$. By \cite[Proposition 2.1]{Suzuki-cpi}, there exists a subset $Z\subseteq X$ of dimension at most $2r-1$ such that $\mult_x(M^2)\le 1$ for all $x\not\in Z$. Let $x\in X\backslash Z$ and let $S$ be a general surface section of $X$ containing $x$, then by \cite[Theorem 0.1]{dFEM-mult-and-lct}, $(S,2M|_S)$ is lc at $x$, hence by inversion of adjunction, $(X,2M)$ is lc at $x$ as well. It follows that for all $0< c \ll 1$, the pair $(X,2(1-c)M)$ is klt outside $Z$. Let $x\in X$ be any point and let $Y\subseteq X$ be cut out by a general linear subspace $V\subseteq\bP^{n+r}$ of codimension $2r-1$ containing $x$. Then $Y\subseteq\bP:=\bP^{n-r+1}$ is also a codimension $r$ complete intersection and we have $K_Y\sim (s+2r-1)H$. Let $D=2M|_Y$ and $L=(s+2r+1)H\sim_\bQ K_Y+D$. Since $V$ is general and $\dim Z \le 2r-1$, $(Y,(1-c)D)$ is klt outside a finite set of points. By Lemma \ref{lem:lct>1/2+delta} we have $\lct(Y;D)>\frac{1}{2}+\delta$ for some absolute constant $\delta>0$ as long as
\begin{equation} \label{eq:h^0(L)}
    h^0(Y,L)\le h^0(\bP,\cO_{\bP}(s+2r+1))=\binom{n+s+r+2}{s+2r+1}<\frac{1}{n-2r+1}\binom{2(n-2r+1)}{n-2r+1}
\end{equation}
which holds as $n\ge N\gg 0$.
\end{proof}

\begin{proof}[Proof of Theorem \ref{thm:Tian's conj}]
This follows directly from Proposition \ref{prop:alpha^2>alpha} and Lemma \ref{lem:alpha^2 of cpi}.
\end{proof}

\begin{proof}[Proof of Corollary \ref{cor:constructible}]
It suffices to show that \eqref{eq:h^0(L)} holds when $s\le -1$ and $n\ge 10r$, which can be easily verified.
\end{proof}

\begin{rem}
For general polarized klt pairs $(X,\Delta;L)$, Conjecture \ref{conj:Tian} is not true: it is not hard to see that the alpha invariant only depends on the numerical equivalence class of the line bundle $L$, but if $D\sim_\bQ L$ is a divisor whose lct computes $\alpha(X,\Delta;L)$ and $N$ is a non-torsion numerically trivial line bundle on $X$, then $D$ is no longer the support of a divisor in some $|m(L+N)|$ and it is unlikely to have another divisor that computes $\alpha(X,\Delta;L+N)$. The following example shows that in general there may not even exist a divisor $D$ numerically equivalent to $L$ such that $\lct(X,\Delta;D)=\alpha(X,\Delta;L)$, so Tian's conjecture is false even up to replacing $L$ by a numerically equivalent one.
\end{rem}

\begin{expl}
Let $C$ be a curve of genus at least $2$ and $\cE$ a stable vector bundle of degree $0$ and rank $2$ on $C$ such that the line bundle $L_0=\cO(1)$ on $X=\bP_C(\cE)$ is pseudo-effective but not numerically equivalent to any effective divisor (see e.g. \cite[Example 1.5.1]{Laz-positivity-1}). Let $\Delta=0$ and $L=L_0+aF$ where $F$ is the fiber class and $a\ge 2$ is an integer. It is clear that $\alpha(X;L)\le \frac{1}{a}$. On the other hand we claim that for any divisor $D\equiv L$ we have $\lct(X;D)> \frac{1}{a}$. To see this, fix a fiber $F$ and let $D=cF+D_0$ where $D_0$ does not contain $F$ in its support, then $c<a$ by our assumption on $L_0$. Since $(D_0\cdot F)=(L\cdot F)=1$, $(F,D_0|_F)$ is lc and thus $(X,F+D_0)$ is also lc by inversion of adjunction. Hence either $c\le 1$ and $(X,D)$ is lc or $1<c<a$ and $(X,\frac{1}{c}D)$ is lc. In both cases, $\lct(X;D)\ge \min\{1,\frac{1}{c}\}>\frac{1}{a}$. Hence $\alpha(X;L)=\frac{1}{a}$ but is not computed by any divisor $D\equiv L$. Note that the same argument also proves that $\alpha^{(2)}(X;L)\ge 1 > \alpha(X;L)$, hence the Picard number one assumption in Proposition \ref{prop:alpha^2>alpha} cannot be removed. Also note that although the alpha invariant is not computed by some $D\equiv L$, it is computed by some divisor $E$ over $X$ in the sense that
\begin{equation} \label{eq:computed by E}
    \alpha(X,\Delta;L)=\frac{A(E;X,\Delta)}{\tau(L;E)}
\end{equation}
where $\tau(L;E)$ is the pseudo-effective threshold of $L$ with respect to $E$ (i.e. the largest $t>0$ such that $\pi^*L-tE$ is pseudoeffective where $\pi:Y\rightarrow X$ is a birational morphism that extracts $E$). Therefore, instead of Conjecture \ref{conj:Tian} it seems better to ask
\end{expl}

\begin{que}
Let $(X,\Delta)$ be a klt pair and $L$ an ample line bundle. Does there always exist a divisor $E$ over $X$ such that \eqref{eq:computed by E} holds?
\end{que} 

\bibliography{ref}
\bibliographystyle{alpha}

\end{document}